\documentclass[12pt]{article}

\usepackage{amsmath}
\usepackage{amssymb}
\usepackage{amsthm}

\title{Generically stable  regular types}
\date{}
\author{Predrag
Tanovi\'c\thanks{Supported by the Ministry of Education,  Science
and Technological Development of Serbia}\\Mathematical Institute
SANU, Belgrade, Serbia}

\newcommand{\nor}{\makebox[1em]{$\not$\makebox[.8em]{$\perp$}}\,}

\def \nmodels {\mathop{\not\models}}
\def \nwor {\mathop{\not\perp}^w}

\def \wor  {\perp\!\!\!{^w}}

\def \tp {{\rm tp}}
\def \Sem {{\rm Sem}}

\def \cl {{\rm cl}}

\def \strok {\upharpoonright}

 \newtheorem{thm}{Theorem}
 \newtheorem{prop}{Proposition}[section]
 \newtheorem{lem}[prop]{Lemma}
 \newtheorem{thm1}[prop]{Theorem}
 \newtheorem{cor}[prop]{Corollary}
  \newtheorem{fact}[prop]{Fact}
  
   \newtheorem{que}{Question}

\theoremstyle{definition}
 
 \newtheorem{defn}[prop]{Definition}
 \newtheorem{rmk}[prop]{Remark}

\begin{document}
\maketitle

\begin{abstract}We study non-orthogonality of symmetric, regular types and show that it preserves generic stability and   is an equivalence relation on the set of all generically stable, regular types. We
will also prove that some of the nice properties from the stable context
hold in  general. In the case of strongly regular types we will relate $\nor$ to the global Rudin-Keisler  order.
\end{abstract}

The concept  of (strong) regularity     for global, invariant
types in an arbitrary first-order theory was  introduced in
Section 3 of \cite{PT}. The definition there was motivated by and
extends that of regular (stationary) and strongly regular types in
stable theories. Intuitively, it can be described as follows: Fix
a global, invariant type $\mathfrak p$. Consider  all the formulas
in $\mathfrak p$ as defining "large" subsets of the monster and
their negations as defining  "small" ones. Then it is natural to
define: $\cl_{\mathfrak p}(X)$ is the union of all small subsets
definable over $X$. It turned out that the regularity of
$\mathfrak p$ means precisely that $\cl_{\mathfrak p}$ is a
closure operation on the locus of $\mathfrak p_{\strok A}$ (for
almost all $A$ over which $\mathfrak p$ is invariant). There are
two kinds of regular types:

\smallskip
- $\mathfrak p$ is symmetric.   Morley sequences are totally indiscernible and
the closure operation  is a pregeometry operation inducing the dimension function.

\smallskip
- $\mathfrak p$ is asymmetric. The  closure operation is induced
by a definable partial ordering which totally orders Morley
sequences.

\smallskip
Stable regular types are symmetric, while asymmetric regular types
may exist only in theories with the strict order property. For example, the type of an infinite element in a theory of dense linear orders without endpoints is strongly regular.
Interesting examples of both kinds of strongly regular types are heirs of
"generic" types of minimal and quasi-minimal groups (and fields);
they were recently studied in \cite{KTW} and \cite{GK}.

Asymmetric
regular types are studied in detail in the forthcoming paper
\cite{MT},  and in this article we will concentrate on non-orthogonality of symmetric, regular
types.   For generically stable regular types we will prove that
non-orthogonality is an equivalence relation.

\begin{thm}\label{Tgenstabnor}
Generic stability is preserved under non-orthogonality of regular
types. Non-orthogonality is an equivalence relation on the set of
all regular, generically stable types.
\end{thm}

Next we will study generically stable, strongly regular types and
prove that non-orthogonality is strongly related to the global
version of Lascar's  Rudin-Keisler order  which was originally
defined in the $\aleph_0$-stable context:

\begin{thm}\label{Tstrregintro}
Suppose that $\mathfrak p$ is generically stable and strongly
regular. Then:

\smallskip (1) \ \  $\mathfrak p$ is RK-minimal in the global RK-order.

\smallskip (2) \ \ If $\mathfrak q$ is invariant then: \ $\mathfrak
p\nor \mathfrak q$ if and only if $\mathfrak p\leq_{RK}\mathfrak
q$.
\end{thm}

The next theorem is probably more surprising than the previous,
because its original proof in the $\aleph_0$-stable case (see e.g
\cite{L2}) relied heavily on the existence of prime models over
arbitrary sets.

\begin{thm}\label{Tlocal}
Suppose that $(\mathfrak{p}(x),\phi_{\mathfrak{p}}(x))$ and
$(\mathfrak{q}(x),\phi_{\mathfrak{q}}(x))$ are $M$-invariant,
strongly regular and generically stable.  Then the following
conditions are all equivalent: \

(1) \ \  $\mathfrak p\perp\mathfrak q$;

\smallskip (2) \ \ $\mathfrak{p}_{\strok M}\wor \mathfrak{q}_{\strok M}$;

\smallskip (3) \ \ For some $C\supseteq M$: \ $\mathfrak{p}_{\strok
C}\wor \mathfrak{q}_{\strok C}$.
\end{thm}

The paper is organized as follows: Section \ref{Spreliminaries}
contains preliminaries. In Section \ref{Sregularity} we will
(slightly) re-define regularity for global invariant types. The
re-definition is needed   due to the fact that in Remark 3.1 in \cite{PT} it was
noted (without proof) that the regularity condition in the
definition does not depend on the particular choice of the parameter set
over which the type is invariant. This is correct if the type is
strongly regular  but we do not know if that holds for an
arbitrary regular type. Fortunately, the lapsus did not affect
proofs of main results, only minor rephrasing of some of the
statements is needed: replacement of "$\mathfrak p$ is regular and
$A$-invariant" by "$\mathfrak p$ is regular over $A$"; the regularity of $\mathfrak p$ over a set is
introduced in Definition \ref{DefPTreg}  below. In Section \ref{Sorth}  we
study non-orthogonality of generically stable regular types and
prove Theorem \ref{Tgenstabnor}. The main technical fact used in the proof, stating that  generically stable regular types have weight one, is proved in Proposition \ref{Pwt1}.   Section \ref{Sstrong}  deals with strongly
regular types and there we prove Theorems \ref{Tstrregintro} and
\ref{Tlocal}.     As an application of the results from Section \ref{Sorth}, in Section \ref{Somit} we prove  that one can vary dimensions of generically stable regular types in countable
models as in the stable case:

\begin{thm}\label{Tomit}
Suppose that $T$ and $A$ are countable and $\{\mathfrak
 p_i \,|\,i\in I\}$ is a countable family of pairwise
orthogonal, regular over $A$, generically stable types. Also assume that each $\mathfrak p_{i\,\strok A}$ is non-isolated. Then for any function $f:I\longrightarrow \omega$ there exists a countable
$M_f\supseteq A$ such that $\dim_{\mathfrak p_i}(M_f/A)=f(i)$ for all $i\in I$.
\end{thm}

\section{Preliminaries}\label{Spreliminaries}

The notation is mainly standard, the only exception is the
convention on the product of invariant types. We fix a complete
first-order theory $T$ and operate in its monster model $\bar M$.
By $a,b,c..., \bar a,\bar b,\bar c,...$ we will denote elements
and tuples of elements,    by $A,B,C,...$     small subsets of the
monster, while $M,M',...$ will denote small elementary submodels.
Global types  will be denoted by $\mathfrak p,\mathfrak
q,\mathfrak r,...$. A global type $\mathfrak p$ is {\em
$A$-invariant} if whenever $ \bar b_1\equiv \bar b_2 \,(A)$ then
$(\phi(\bar x,\bar b_1)\Leftrightarrow \phi(\bar x,\bar b_2))\in
\mathfrak p(\bar x)$ for all $\phi(\bar x;\bar y)$ with parameters
from $A$. $\mathfrak p$ is {\em invariant} if it is $A$-invariant
over some small $A$. $\mathfrak p_{\strok A}$ will denote the
restriction of $\mathfrak p$ to $A$ and $(a_i\,|\,i<\alpha)$ is a
{\em Morley sequence in $\mathfrak p$ over $A$} if $a_i\models
\mathfrak p_{\strok A a_{<i}}$ for all $i<\alpha$.

Assume for a while that  $\mathfrak p$ is $A$-invariant.  Then
Morley sequences in $\mathfrak p$ over $A$ are indiscernible. We
will occasionally go out of $\bar M$ (into a larger monster) in
order to get realizations of global types; these will be also
denoted by $\bar a, \bar b,...$, in which case $\mathfrak p_{\strok
\bar M\bar a}$ will be well-defined due to the invariance of
$\mathfrak p$. Thus global Morley sequences are also well-defined,
as well as the powers   $\mathfrak p^{\alpha}$  (types of Morley
sequences of length $\alpha$) are. Let $\bar a_1,\bar a_2 \models
\mathfrak p^2$. If $ \tp(\bar a_1,\bar a_2/\bar M)=\tp(\bar
a_2,\bar a_1/\bar M)$ then we will say that $\mathfrak p$ is {\em
symmetric}; otherwise, it is {\em asymmetric}. If $\mathfrak p$ is
symmetric,  $(\bar a_i\,|\,i<\alpha)\models \mathfrak p^{\alpha}$ and
$\pi$ is a permutation of $\alpha$ then
$(\bar a_{\pi(i)}\,|\,i<\alpha)\models  \mathfrak p^{\alpha}$.

Products of invariant types were introduced in \cite{HP}.  Here we
will reverse the  order in the definition: if $\mathfrak p$ and
$\mathfrak q$ are invariant then their  {\em product}  $\mathfrak
p(\bar x)\otimes\mathfrak q(\bar y)$ is defined as follows: if
$\bar a\models \mathfrak p$ and $\bar b\models \mathfrak q_{\strok
\bar M\bar a}$ then $\mathfrak p(\bar x)\otimes\mathfrak q(\bar
y)=\tp_{\bar x,\bar y}(\bar a,\bar b/\bar M)$; thus our $\mathfrak p(\bar x)\otimes\mathfrak q(\bar
y)$ is the original $\mathfrak q(\bar y)\otimes\mathfrak p(\bar
x)$. This change was suggested by Ludomir Newelski due to the fact  that it is
natural to have the equivalence: \ $(\bar a_1,\bar a_2)$ is a Morley sequence
in $\mathfrak p$ over $A$ \ if and only if \ $\bar a_1,\bar a_2\models
\mathfrak p\otimes\mathfrak p$\,. The product is associative, but not commutative. We say that $\mathfrak p$ and $\mathfrak q$ commute if $\mathfrak p(\bar x)\otimes\mathfrak q(\bar
y)= \mathfrak q(\bar y)\otimes\mathfrak p(\bar
x)$.

\smallskip Complete types $p,q$ over the same domain are {\em weakly
orthogonal}, or $p\wor q$, if $p(\bar x)\cup q(\bar y)$ determines
a complete type. Global types $\mathfrak p$ and $\mathfrak q$ are {\em
orthogonal}, or $\mathfrak p\perp \mathfrak q$,  if they are
weakly orthogonal. It is possible that both $\mathfrak p\nor
\mathfrak q$ and $\mathfrak p_{\strok A}\wor \mathfrak q_{\strok A}$
hold for $A$-invariant types (even regular in a superstable
theory). The opposite situation,   $\mathfrak p\perp \mathfrak q$
and $\mathfrak p_{\strok A}\nwor \mathfrak q_{\strok A}$    may occur in
an unstable theory, but not in a stable one.  In a  stable theory
$\mathfrak p_{\strok A}\nwor\mathfrak q_{\strok A}$ implies $\mathfrak
p_{\strok C}\nwor\mathfrak q_{\strok C}$ for all $C\supseteq A$ and, in
particular, $\mathfrak p\nor\mathfrak q$. We will see in
Proposition \ref{Pwor=or} that this holds for any regular, generically stable type. For definable types over a model we have:

\begin{fact}\label{For}
Suppose that  $\mathfrak p$ and $\mathfrak q$ are both
$M$-invariant and   definable, and    $\mathfrak p_{\strok M}\nwor
\mathfrak q_{\strok M}$. Then  $\mathfrak p_{\strok C}\nwor \mathfrak
q_{\strok C}$ for all $C\supseteq M$; in particular, $\mathfrak p\nor
\mathfrak q$.
\end{fact}

\smallskip
A non-algebraic global type $\mathfrak{p}(\bar{x})$ is {\em generically
stable} if, for some small $A$,  it is $A$-invariant
and:
\begin{center}
if $\alpha$ is infinite and  $(\bar{a}_i\,:\,i<\alpha)$   is a
Morley sequence in $\mathfrak{p}$ over $A$ then for any formula
$\phi(\bar{x})$ (with parameters from $\bar{M}$) \ $\{i\,:\
\models\phi(\bar{a}_i)\}$ is either finite or
co-finite.\end{center} 
Using compactness it is straightforward to
check that $\mathfrak p$ is generically stable if the condition
holds for $\alpha=\omega+\omega$. Also, if $\mathfrak{p}$ is
generically stable then as a witness-set $A$ in the definition we
any small $A$ over which $\mathfrak{p}$ is  invariant can be taken.
Generically stable types are definable and symmetric. They commute
with all invariant types. A power of a generically stable type may
not be generically stable;   an example the reader can find in
\cite{ACP}. However, this cannot happen if $\mathfrak p$ is in
addition regular.

\smallskip
Let $\mathcal C$ be any subset of the monster. A  partial type {\em
$\pi(x)$ is finitely satisfiable in $\mathcal C$} if any finite
subtype has infinitely many realization in $\mathcal C$, in which case  we  also
say that $\pi(x)$ is a {\em $\mathcal C$-type}. By a {\em
$\mathcal C$-sequence over $A$} we will mean a sequence
$(a_i\,|\,i\leq n)$ such that $\tp(a_i/A\bar a_{<i})$ is a $\mathcal
C$-type for each $i\leq n$.  Elements of a $\mathcal C$-sequence
can realize distinct types because there may exist many distinct
$\mathcal C$-types, so a $\mathcal C$ -sequence  may not be
indiscernible. The following well known fact  guaranties existence
of $\mathcal C$-extensions.

\begin{fact} \label{Fext}
Suppose that  a   partial type $\pi(x)$ is defined over $A$ and
finitely satisfiable in $\mathcal C$. Then for any $B\supseteq A$
there exists a $\mathcal C$-type  in $S(B)$ extending $\pi(x)$.
\end{fact}

Non-isolation of  $p\in S_n(A)$ can be expressed in terms of
satisfiability: Fix $\phi(x)\in p$ and let $\mathcal C=\phi(\bar
M)\smallsetminus p(\bar M)$. Then $p$ is non-isolated if and only
if it is a $\mathcal C$-type. Thus, isolation of a type is a strong negation of its finite satisfiability. The next fact follows from Fact \ref{Fext}.

\begin{fact} \label{Ffs}
Suppose that $p\in S_1(A)$ is non-isolated, $\phi(x)\in p$,  and
$A\subseteq B$. Then $p$ has an extension in $S_1(B)$ which is
finitely satisfiable in $\mathcal C=\{c\in\bar M\,|\,
\phi(x)\in\tp(c/A)\neq p\}$.
\end{fact}

A weak negation of satisfiability is the semi-isolation:  $\bar
b$ is {\em semi-isolated by} $\bar a$ over $A$ (or $\bar a$ {\em
semi-isolates $\bar b$} over $A$),    denoted also by $\bar b\in
\Sem_A(\bar a)$,\, iff \,there is a formula $\phi(\bar y,\bar
x)\in \tp(\bar b,\bar a/A)$ such that $\phi(\bar a,\bar x)\vdash
\tp(\bar b/A)$; $\phi(\bar x,\bar y)$ is said to witness the
semi-isolation. Semi-isolation is   transitive: if $\bar
b\in\Sem_A(\bar a)$ is witnessed by $\phi(\bar y,\bar a)$ and
$\bar c\in\Sem_A(\bar b)$ is witnessed by $\psi(\bar z,\bar b)$,
then $\bar c\in\Sem_A(\bar a)$ is witnessed by $\exists \bar
y(\phi(\bar y,\bar a)\wedge\psi(\bar z,\bar y))$.

\begin{fact} \label{Fsemifs}
Suppose that $\tp(  a/A)=p$ is non-algebraic and  $\phi(
x)\in p$. Let $\mathcal C=\{  c\in\bar M\,|\, \phi(
x)\in\tp(  c/A)\neq p\}$ and assume  $\mathcal C\neq
\emptyset$. Then:

\smallskip (1) $\psi(  x,\bar b)\in\tp(  a/A\bar b)$
witnesses $  a\in\Sem_A( \bar b)$ if and only if it is not
satisfied in $\mathcal C$.

\smallskip (2) $  a\notin\Sem_A(\bar b)$ \   if and only if  \
$\tp(  a/A\bar b)$ is a  $\mathcal C$-type.
\end{fact}

The Rudin-Keisler order on complete types was introduced by Lascar
in \cite{L1}. In an $\aleph_0$-stable theory it was related to
strong regularity; a nice exposition of the material can be found in Lascar's book
\cite{L2},  or in Poizat's book \cite{Po}. For non-algebraic types
$p,q\in S(M)$ define $p\leq_{RK} q$ iff every model $M'\supset M$
which realizes $p$ also realizes $q$. $\leq_{RK}$ is a quasi-order
which particularly well behaves in an $\aleph_0$-stable theory; there,
due to the existence of prime models over arbitrary sets, omitting
types is much easier than in general. In an $\aleph_0$-stable
theory  RK-minimal elements exist  and they are precisely the strongly
regular types. Some of equivalent ways of defining $p\leq_{RK} q$
in the $\aleph_0$-context are:
\begin{enumerate}
\item $p$ is realized  in $M(q)$ (the model prime over $M$ and a realization of $q$);

\item    There are $\bar a\models p$ and $\bar b \models q$ such that $\tp(\bar a/M\bar b)$ is isolated;

\item  There are $\bar a\models p$ and $\bar b \models q$ such that $  \bar a\in \Sem_M(\bar b)$.
\end{enumerate}
In this article we will consider a variant of the RK-order,
defined only for global types. In the unstable  context   only the
third  equivalent is adequate.

\begin{defn}\label{Drkorder} \ $\mathfrak
p\leq_{RK}\mathfrak q$ \ if and only if  there are $\bar a\models
\mathfrak p$ and $\bar b \models \mathfrak q$ such that $\bar a\in \Sem_{\bar M}(\bar
b)$.
\end{defn}

Transitivity of semi-isolation implies that we have a
quasi-order. However, the order  can be quite trivial: take the
theory of the random graph and notice that no two distinct
global 1-types are $\leq_{RK}$-comparable.

\smallskip
If both $\mathfrak p\leq_{RK}\mathfrak q$ and $\mathfrak
q\leq_{RK}\mathfrak p$ hold then we say that $\mathfrak p$ and
$\mathfrak q$ are {\em RK-equivalent} and  denote it   by $\mathfrak
p\sim_{RK}\mathfrak q$. $\mathfrak p$ and $\mathfrak q$ are {\em
strongly RK-equivalent}, or $\mathfrak p\equiv_{RK}\mathfrak q$,
if there are $\bar a\models \mathfrak p$ and $\bar b\models
\mathfrak q$ such that both $\bar a\in\Sem_{\bar M}(\bar b)$ and
$\bar b\in\Sem_{\bar M}(\bar a)$ hold.   RK-equivalent types   may
not be strongly RK-equivalent.

\section{Regularity}\label{Sregularity}

In this section we will re-define regularity for global invariant
types. To simplify notation we  define it for global 1-types only.
This will not affect the generality because we can always switch
to an appropriate sort in $\bar M^{eq}$.  The definition given
here slightly differs in that we first define when $\mathfrak p$ is regular over $A$ (here $\mathfrak p$ is $A$-invariant), and then
repeat the original one: $\mathfrak p$ is regular if such a small
set $A$ exists. Concerning strong regularity, the definition
remains unchanged.

\begin{defn}\label{DefPTreg}
Let $\mathfrak p(x)$ be  a global non-algebraic  type and let $A$ be small.

\smallskip
(i) \ $\mathfrak p(x)$ is said to be \emph{regular over $A$} if  it is $A$-invariant  and for any $B\supseteq A$
and $a\models \mathfrak p_{\strok A}$: \ either $a\models \mathfrak p_{\strok B}$ \
or \ $\mathfrak p_{\strok B}\vdash \mathfrak p_{\strok Ba}$.

\smallskip
(ii) \ $\mathfrak p$ is {\em regular} if $\mathfrak p$ is regular over some small set.
\end{defn}

Clearly, if $\mathfrak p$ is regular over $A$  and $A\subseteq B$ then
$\mathfrak p$ is regular over $B$, too.  The same observation holds
for strong regularity. But, before defining  strong regularity  it
is convenient to introduce the following notation: we will say
that $(\mathfrak p(x),\phi(x))$ is an {\em $A$-invariant pair} if
$\mathfrak p$ is $A$-invariant and $ \phi(x)\in \mathfrak
p_{\strok A}$.

\begin{defn}
(i)   $(\mathfrak p(x),\phi(x))$ is \emph{strongly regular} if for
some small $A$  it is an $A$-invariant pair and:
\begin{center}
   for all $B\supseteq A$ and $a$ satisfying $\phi(x)$: \ either
$a\models \mathfrak pp_{\strok B}$   or    $\mathfrak pp_{\strok B}\vdash
\mathfrak pp_{\strok Ba}$.
\end{center}

(ii) \ $\mathfrak p$ is {\em strongly regular} if $(\mathfrak p,\phi(x))$ is strongly regular for some $\phi(x)\in \mathfrak p$.
\end{defn}

We will prove in Proposition \ref{PequivSR} that as a witness set
$A$ in the previous definition we can take any small $A$ for which
$(\mathfrak p,\phi)$ is $A$-invariant. For, we need to label a
local regularity condition.

\begin{defn}\label{Dwor}
Suppose that  $p\in S(A)$ and $\pi\subseteq p$.  We say that
$(p,\pi)$ satisfies {\em the weak orthogonality condition}, or
(WOR) for short, if:
\begin{center}
$p\,\wor\, \tp(\bar b/ A)$ \ \ \ for all
$\bar{b}\subset\pi(\bar{M})\setminus p(\bar M)$.\end{center}
\end{defn}

WOR is a technical property of locally strongly regular types, see
Definition 7.1 in \cite{PT}. Examples of such types are "generic"
types of minimal and quasi-minimal structures. Recall that $M$ is
a minimal structure iff any definable subset with parameters is
either finite or co-finite. In  a minimal structure  there is a unique
non-algebraic type  $p\in S_1(M)$. 
$(p(x),x=x)$ satisfies WOR,  so $p$ is locally strongly regular via $x=x$.
By Corollary 7.1 from \cite{PT} the same is true if $p$ is the
"generic" type of a quasi-minimal structure (the type containing
all the formulas with a co-countable solution set).

\begin{rmk}\label{Rwor}
  Some of equivalent ways  of expressing the fact  that $(p,\pi)$
satisfies WOR are: \ \
\begin{enumerate}
\item \ \ \  $p\,\vdash p\,|\,A\bar{b}$ \ \  for all
$\bar{b}\subset\pi(\bar{M})\smallsetminus p(\bar M)$

\item \ \ \ $p\vdash p\,|A\cup(\pi(\bar M)\smallsetminus p(\bar
M))$

\end{enumerate}
\end{rmk}

\begin{lem}\label{Lworinv}
Suppose that $(\mathfrak  p(x), \phi(x))$ is   $A$-invariant,
$A\subseteq B$,    and   $(\mathfrak p_{\strok B},\phi)$ satisfies
WOR. Then $(\mathfrak p_{\strok A},\phi)$ satisfies WOR, too.
\end{lem}
\begin{proof}
Suppose, on the contrary, that   $\bar c \subset \phi(\bar
M)\smallsetminus \mathfrak p_{\strok A}(\bar M)$  is such that
$\mathfrak p_{\strok A}\nvdash \mathfrak p_{\strok A\bar c}$. Then
there are $a',a''$ realizing $\mathfrak p_{\strok A}$ and a formula
$\varphi(x,\bar y)$ (over $A$) such that $\models \varphi(a',\bar
c)\wedge\neg\varphi(a'',\bar c)$. \ Let $B',B''$ be such that
$B\equiv B'\equiv B''\,(A)$, $a'\models \mathfrak p_{\strok B'}$ and
$a''\models \mathfrak p_{\strok B''}$. \ Note that the
$A$-invariance of $\mathfrak  p$ implies that  both $(\mathfrak p_{\strok B'},\phi)$ and $(\mathfrak p_{\strok B''},\phi)$ satisfy WOR.
Since $\bar c\in \phi(\bar M)\smallsetminus \mathfrak p_{\strok A}(\bar
M)$ both  $\mathfrak p_{\strok B'}\vdash \mathfrak p_{\strok B'\bar c}$
and $\mathfrak p_{\strok B''}\vdash \mathfrak p_{\strok B''\bar c}$
hold. In particular:
\begin{center}
$\mathfrak p_{\strok B'}(x)\vdash \varphi(x,\bar c)$ \ \ and  \ \
$\mathfrak p_{\strok B''}(x)\vdash \neg\varphi(x,\bar c)$
\end{center}
Let $a$ realize $\mathfrak p_{\strok B'B''}$. Then $\models \varphi(a,\bar
c)\wedge\neg\varphi(a,\bar c)$. A contradiction.
\end{proof}

\begin{prop}\label{PequivSR}
(i) An $A$-invariant pair   $(\mathfrak{p}(x),\phi(x))$ is
strongly regular if and only if  $(\mathfrak p_{\strok B},\phi)$
satisfies WOR for all $B\supseteq A$.

\smallskip(ii) An $A$-invariant pair   $(\mathfrak{p}(x),\phi(x))$ is
strongly regular if and only if: for all $B\supseteq A$ and $a$
satisfying $\phi(x)$: \ either $a\models \mathfrak p\,|\,B$   or
$\mathfrak p_{\strok B}\vdash \mathfrak p_{\strok Ba}$. Therefore,
as a witness set $A$ in the definition of strong regularity we can
take any small set $A$ over which $(\mathfrak p,\phi)$ is
invariant.
\end{prop}
\begin{proof}(i)
$\Leftarrow$) is easy, so we prove only $\Rightarrow$). Suppose
that $(\mathfrak{p}(x),\phi(x))$ is strongly regular. Let $A_1\supseteq  A$
be  such that the regularity condition holds:
\begin{center} for all $B_1\supseteq A_1$ and $a$ satisfying
$\phi(x)$: either \ $a\models \mathfrak p_{\strok B_1}$ \
or \  $\mathfrak p_{\strok B_1}\vdash \mathfrak p_{\strok B_1a}$.
\end{center}
Fix $B_1\supseteq A_1$   and we will show that $(p,\phi)$
satisfies WOR (where $p=\mathfrak p_{\strok B_1}$).  Suppose that
$b_1...b_n=\bar b\subset \phi(\bar M)\smallsetminus p(\bar M)$.
Apply the regularity condition to $(b_1, B_1)$: $b_1\in \phi(\bar
M)\smallsetminus \mathfrak p_{\strok B_1}(\bar M)$ so $ p(x)\vdash
\mathfrak p_{\strok B_1b_1}(x)$; then apply it to $(b_2,B_1b_1)$: $b_2\in
\phi(\bar M)\smallsetminus \mathfrak p_{\strok B_1}(\bar M)$ so $
\mathfrak p_{\strok B_1b_1}(x)\vdash \mathfrak p_{\strok B_1b_1b_2}(x)$;
continuing in this way we get:
$$ p(x)\vdash \mathfrak p_{\strok B_1b_1}(x)\vdash \mathfrak p_{\strok B_1b_1b_2}(x) \vdash ...\vdash \mathfrak p_{\strok B_1b_1...b_{n}}(x)$$
Thus $p(x)\vdash \mathfrak p_{\strok B_1\bar b}(x)$ and $(p,\phi)$ satisfies WOR. Now let $B\supseteq A$. Then $(\mathfrak p_{\strok BA_1},\phi)$ sastisfies WOR so, by Lemma  \ref{Lworinv}, $(\mathfrak p_{\strok B},\phi)$ satisfies WOR, too.

\smallskip
(ii) Follows from  part (i).
\end{proof}

\begin{rmk}\label{Rworreg}
Suppose that  $\mathfrak{p}(x)$ is   non-algebraic and
$A$-invariant.  As in the proof of Proposition \ref{PequivSR}(i) one checks that $\mathfrak p$ is regular over $A$\ if 
only if \ $\mathfrak p$ is $A$-invariant and $(\mathfrak p_{\strok
B},\mathfrak p_{\strok A})$ satisfies WOR for any (finite) extension
$B\supseteq A$. The stronger   equivalence, like the  one that we
have  established for strongly  regular types in Proposition
\ref{PequivSR},  would be: an $A$-invariant type is regular iff
$\mathfrak p$ is regular over $A$. However that does not seem to hold:
it is likely that there is a regular, $A$-invariant type which is not regular over $A$ (but we don't know of an
example).
\end{rmk}

The following fact, suggested  by Anand Pillay, shows that the
stronger equivalence  holds for regular types under additional
assumptions:

\begin{prop}
Suppose that  $\mathfrak{p}(x)$ is  definable and $M$-invariant.
Then $\mathfrak p$ is regular if and only if it is regular over $M$.
\end{prop}
\begin{proof}Only
$\Rightarrow$) requires a proof. Suppose that $\mathfrak p$ is
regular and  let $A\supseteq M$ be such that $\mathfrak p$ is regular over $A$. We claim that $\mathfrak p$ is regular over $M$, too.
Otherwise, for some  $\bar c$  there are $a,b$ realizing $\mathfrak
p_{\strok M}$   such that $a\models \mathfrak p_{\strok M\bar c}$,
$b\nmodels \mathfrak p_{\strok M\,\bar c}$ and $a\nmodels \mathfrak
p_{\strok M\bar cb}$. Let $A_1\equiv A\,(M)$ be such that
$\tp(a,b,\bar c/A_1)$ is a heir of $\tp(a,b,\bar c/M)$. Then  both
$a$ and $b$ realize  $\mathfrak p_{\strok A_1}$, because $\mathfrak p$ is a heir of  $\mathfrak p_{\strok M}$.

Suppose $a\nmodels \mathfrak p_{\strok A_1\,\bar c}$ and find $\bar
d\in A_1$ and $\varphi$ such that $\models \varphi(a,\bar c,\bar
d)\wedge d_{\mathfrak p}\neg\varphi(t,\bar c,\bar d)$.   Since
$\tp(\bar d/\bar caM)$ is a coheir, there exists $\bar d'\in M$
such that $\models \varphi(a,\bar c,\bar d')\wedge d_{\mathfrak
p}\neg\varphi(t,\bar c,\bar d')$; hence $a\nmodels\mathfrak
p_{\strok M\bar c}$. A contradiction. We conclude $a\models \mathfrak
p_{\strok A_1\,\bar c}$.

Therefore $a\models \mathfrak p_{\strok A_1\,\bar c}$,\, $b\models
\mathfrak p_{\strok A_1}$ and $a\nmodels \mathfrak p_{\strok
A_1\,b\,\bar c}$, so  $\mathfrak p$ is not regular over $A_1$. A
contradiction.
\end{proof}

Suppose that $\mathfrak p$ is $A$-invariant and let $p=\mathfrak
p_{\strok A}$.  Define $\cl_{\mathfrak p,A}$ as an operation on the
power set of $p(\bar M)$:
\begin{center}
$\cl_{\mathfrak p,A}(X)=\{ b\in p(\bar M)\,|\,b\nmodels \mathfrak
p_{\strok A\,X} \}$ \ \ for all $X\subset p(\bar M)$.
\end{center}
If $\mathfrak p$ is regular over $A$ then, by Lemma 3.1(iii) from
\cite{PT}, $\cl_{\mathfrak p,A}$ a closure operator on $p(\bar
M)$. The proof of this fact does not depend on Remark 3.1 there,
neither does the proof of Theorem 3.1 there, which is a  dichotomy
theorem for regular types. Here we state only a restricted version
and we will use only the first part:

\begin{thm1}
Suppose that  $\mathfrak p$ is regular over $A$. Then
$\cl_{\mathfrak p,B}$ is a closure operator on $\mathfrak p_{\strok
B}(\bar M)$ for all $B\supseteq A$. We have two kinds of regular types:

\smallskip
(1) \ Symmetric ($\mathfrak p$ is symmetric). Then $\cl_{\mathfrak
p,B}$ is a pregeometry operator on $p(\bar{M})$ for all $B\supseteq A$.

\smallskip
(2) \ Asymmetric. Then   there exists  a finite extension $A_0$ of
$A$ and an  $A_0$-definable  partial order $\leq$ such that every
Morley sequence in $p$ over $A_0$ is strictly increasing; $\cl_{\mathfrak
p,A_0}$ is not a pregeometry.
\end{thm1}

In this paper we will deal only with  symmetric regular types.
Then the pregeometry describes the independence:
$(a_i\,|\,i\in\alpha)$ is a Morley sequence in $\mathfrak p$ over
$A$ if and only if it is $\cl_{\mathfrak p,A}$-independent. In
particular, maximal Morley sequences in any $M\supseteq A$ have
the same cardinality, so $\dim_{\mathfrak p}(M/A)$ is a well
defined cardinal number.

\section{Orthogonality}\label{Sorth}

In this section we study orthogonality of regular symmetric types.
Our goal is to prove Theorem \ref{Tgenstabnor}. We start by
mentioning a result from \cite{MT}; it will not be  used further in
the text:

\begin{thm1}
A regular asymmetric type is orthogonal to any symmetric invariant
type. In particular, symmetry is preserved under non-orthogonality
of regular types.
\end{thm1}

\begin{que}
Is $\nor$ an equivalence relation on the set of all
regular symmetric types?
\end{que}

Below, a positive answer will be given  for generically stable types.

\begin{lem}\label{L1} Suppose that $\mathfrak{p}$ and $\mathfrak{q}$ are $A$-invariant and that
$\mathfrak p$ is regular over $A$ and  symmetric. Also  suppose that
$\bar b\models \mathfrak{q}_{\strok A}
=q$ and  $a\models\mathfrak p_{\strok A}=p$ are such that $a$ does not
realize $\mathfrak{p}_{\strok A\bar b}$ and let  $\phi(x,\bar
b)\in\tp(a/A\bar b)\smallsetminus\mathfrak{p}_{\strok A\bar b}$.
Then exactly one  of the following two conditions holds:

\smallskip  (A) \ Whenever $\{ a_i
\,|\,i\in\omega+\omega\}$ is a Morley sequence in $\mathfrak{p}$
over $A$ then   \begin{center} $q(\bar y)
\cup\{\phi(a_i,\bar y)\,|\,i\in\omega\}\cup \{\neg\phi
(a_{\omega+i},\bar y)\,|\,i\in\omega\} \ \ \textmd{is
consistent.}$\end{center}

\smallskip  (B) \    $p(x)\cup\mathfrak{q}^2_{\strok A}(\bar y_1,\bar y_2)\cup \{\phi(x,\bar
y_1)\wedge \phi(x,\bar y_2)\}$ \ is inconsistent.

\smallskip\noindent  Moreover, if $\mathfrak p$ is generically stable then (B) holds.

\end{lem}
\begin{proof}
Suppose that neither (A) nor  (B) are satisfied  and work for a contradiction. The failure of (A), by compactness, implies that for some $n$
\begin{equation}\label{e1.1} q(\bar y)  \cup\{\phi(a_i,\bar y)\,|\,1\leq i\leq
n\}\cup \{\neg\phi (a_{j},\bar y)\,|\, n< j \leq 2\,n\} \ \
\textmd{is inconsistent} \, ;\end{equation}
We {\em claim} that  $q(\bar y)
\cup\{\phi(a_i,\bar y)\,|\,i\leq n\}$ is inconsistent.
Otherwise it would be satisfied by some $\bar b'$ and whenever
$(a_{n+1}',\ldots, a_{2n}')$ is a Morley sequence in $\mathfrak p$
over $A\,\bar b'a_{\leq n}$ we would have $\models
\bigwedge_{i}\neg\phi(a_i',\bar b')$; this  is justified by $\neg\phi(x,\bar b')\in \mathfrak p$   which is implied  by:  $\neg\phi(x,\bar b)\in \mathfrak p$, the $A$-invariance of $\mathfrak p$, and $\tp(\bar b/A)=\tp(\bar b'/A)=q$. Therefore $\bar b'$ realizes
$$ q(\bar y)  \cup\{\phi(a_i,\bar y)\,|\,i\leq
n\}\cup \{\neg\phi (a_{j}',\bar y)\,|\, n< j \leq 2\,n\} \,$$which
is in contradiction with (\ref{e1.1}).

\smallskip
Let $n_{\phi}$ be   maximal   such that $q(\bar y)
\cup\{\phi(a_i,\bar y)\,|\,i<n_{\phi}\}$ is consistent and,
without loss of generality, assume that $\bar b$ realizes the
type. The  maximality of $n_{\phi}$ implies that no element of
$\phi(\bar M,\bar b)\cap p(\bar M)$ realizes $\mathfrak{p}_{\strok A
\bar c}$, where $\bar c$ denotes $a_0\ldots a_{n_{\phi}-1}$.  Hence
$\phi(\bar M,\bar b)\cap p(\bar
M)\subseteq\cl_{\mathfrak{p},A}(\bar c)$.

\smallskip
Fix $m> n_{\phi}$ and let  $\bar c_0,\bar c_1,\ldots,\bar c_{m}$
be a Morley sequence in $\mathfrak{p}^{n_{\phi}}_{\strok A}$. For
each $i\leq m$ choose $\bar b_i$   such that $\bar c_i\bar
b_i\equiv \bar c\bar b\,(A)$; note that $\phi(\bar M,\bar b_i)\cap
p(\bar M)\subseteq\cl_{\mathfrak{p},A}(\bar c_i)$ holds. Let $\bar
b'$ realize $\mathfrak{q}_{\strok A\bar c_{\leq m}\bar b_{\leq m}}$.
$(\bar b',\bar b_i)$ is a Morley sequence in $\mathfrak{q}$ over
$A$ for  each $i\leq m$, so the failure of (B) implies that \
$p(x) )\cup \{\phi(x,\bar b')\wedge \phi(x,\bar b_i)\}$ \ is
consistent;  let $d_i$ realize  it. Then $\models \phi(d_i,\bar
b_i)$ implies $d_i\in \cl_{\mathfrak{p},A}(\bar c_i)$. Since $\mathfrak p$ is symmetric, $\cl_{\mathfrak{p},A}$ is a pregeometry so the
$\cl_{\mathfrak{p},A}$-independence of $\bar c_i$'s implies that
$d_0,d_1,\ldots,d_{m}$ is a Morley sequence in $\mathfrak{p}_{\strok
A}$. Thus $\bar b'$ realizes \ $q(\bar y) \cup\{\phi(d_i,\bar
y)\,|\,i\leq m\}$;  this contradicts the maximality of $n_{\phi}$.
\end{proof}

In the next proposition we will  prove that generically stable
regular types "have weight 1 with respect to the $\otimes$-independence".

\begin{prop}\label{Pwt1}
Suppose that $\mathfrak{p}$,  $\mathfrak{q}$ and $\mathfrak{r}$
are (not necessarily distinct) $A$-invariant types and that
$(\mathfrak{p},A)$ is regular and generically stable. Let $\bar b\,\bar
c\models\mathfrak{q}\otimes\mathfrak{r}_{\strok A}$ and
$a\models \mathfrak{p}_{\strok A}$. Then at least one of
$a\models\mathfrak{p}_{\strok A\bar b}$ and
$a\models\mathfrak{p}_{\strok A\bar c}$ holds.
\end{prop}
\begin{proof}
Suppose that neither of them holds and choose $\varphi(x,\bar b)\in\tp(a/A\bar
b)\smallsetminus \mathfrak p_{\strok A\bar b}$ and $\psi(x,\bar
c)\in\tp(a/A\bar c)\smallsetminus \mathfrak p_{\strok A\bar c}$. Let
$\phi(x,\bar b,\bar c)$ be $\varphi(x,\bar b)\vee\psi(x,\bar c)$.
Let $\mathfrak q'= \mathfrak{q}\otimes\mathfrak{r}$, $\bar b'=\bar b\,\bar c$. Then we have:

\smallskip - \ $\mathfrak{p}$ and $\mathfrak{q'}$ are $A$-invariant and
$\mathfrak{p}$ is  regular and  symmetric;

\smallskip - \   $a\models\mathfrak p_{\strok A}$ and $\bar b'\models \mathfrak{q'}_{\strok
A}$;

\smallskip - \ $a$ does not
realize $\mathfrak{p}_{\strok A\bar b'}$ and    $\phi(x,\bar
b')\in\tp(a/A\bar b')\smallsetminus\mathfrak{p}_{\strok A\bar b'}$.

\smallskip\noindent
Therefore  $\mathfrak{p},\mathfrak{q'}, a,\bar b'$ and
$\phi(x,\bar b')$ satisfy assumptions of Lemma \ref{L1} in place
of $\mathfrak p,\mathfrak q, a, \bar b$.  Since $\mathfrak p$ is
generically stable option (B) holds:
\setcounter{equation}{0}\begin{equation}\label{EB} \mathfrak
p_{\strok A} (x)\cup(\mathfrak{q'})^2_{\strok A}(\bar y_1,\bar
y_2)\cup \{\phi(x,\bar y_1)\wedge \phi(x,\bar y_2)\} \ \textmd{ is
inconsistent.}
\end{equation}
 Let
$\bar b\bar c, \bar b_1\bar c_1$ be a Morley sequence in
$\mathfrak{q}\otimes\mathfrak{r}$ over $A$ and let $a_1$ be such
that $a_1\,\bar c_1\equiv a\,\bar c\,(A\,\bar b)$. Then $\models
\varphi(a_1,\bar b)$ implies $\models \phi(a_1,\bar b,\bar c)$\,, and
$\models \psi(a_1,\bar c_1)$ implies $\models \phi(a_1,\bar b_1,\bar
c_1)$. Summing up, we have:
\begin{center}
$a_1\models \mathfrak{p}_{\strok A}(x)\,, \ \
(\mathfrak{q'}^2)_{\strok A}(\bar b\,\bar c,\bar b_1\,\bar c_1) \
\ \textmd{and}\ \   \models \phi(a_1,\bar b,\bar c)\wedge
\phi(a_1,\bar b_1,\bar c_1) \  $\end{center} which contradicts
(\ref{EB}). This proves the proposition.
\end{proof}

\begin{cor}\label{Cwt1}
Suppose that $\mathfrak{p}$,  $\mathfrak{q}$ and $\mathfrak{r}$
are  invariant   and that $\mathfrak{p}$ is regular and
generically stable. Then \  $\mathfrak p\perp\mathfrak
q\otimes\mathfrak r$ \ if and only if \   $\mathfrak
p\perp\mathfrak q$ and $\mathfrak p\perp\mathfrak r$.
\end{cor}

\begin{lem}\label{L16}
Suppose that $\mathfrak{p,q,r}$ are $A$-invariant, $\mathfrak p$ and $\mathfrak q$ are  regular over $A$,   and
that $\mathfrak{p}$ is   generically stable. Further, suppose that
$a,b,\bar c$ are realizations of $p,q,r $ (where $x=\mathfrak
x_{\strok A}$) respectively  such that   $b\nmodels
\mathfrak{q}_{\strok A\bar c}$  and $a\nmodels
\mathfrak{p}_{\strok Ab}$. Then:

\smallskip (i)  \ $a\nmodels \mathfrak{p}_{\strok A\bar c}$.

\smallskip (ii)   For all $\phi(y,\bar c)\in \tp(b/A\bar c)$
witnessing $b\nmodels \mathfrak{q}_{\strok A\bar c}$ and
$\theta(b,x)\in\tp(a/Ab)$ witnessing  $a\nmodels
\mathfrak{p}_{\strok Ab}$  there exist
$\varphi_{\mathfrak{q}}(y)\in\mathfrak{q}_{\strok A}$ and
$\varphi_{\mathfrak{p}}(x)\in \mathfrak{p}_{\strok A}$ such that
$$ \exists y( \varphi_{\mathfrak{p}}(x) \wedge
\varphi_{\mathfrak{q}}(y) \wedge \phi(y,\bar c) \wedge
\theta(y,x))\notin \mathfrak{p}_{\strok A\bar c}(x)\ .$$
\end{lem}
\begin{proof}
(i)  Suppose on the contrary that  $a\models \mathfrak{p}_{\strok A\bar c}$.
Let   $b'$ realize $\mathfrak{q}_{\strok A\bar c}$ and let $a'$ be a realization of $p$ such that
$ab\equiv a'b'\,(A)$.
We {\em claim} that $a'\models \mathfrak{p}_{\strok A\bar c}$ holds. Otherwise we have
$$   a'\nmodels \mathfrak{p}_{\strok A\bar c} \ \ a'\nmodels \mathfrak{p}_{\strok Ab'} \ \ \textmd{and} \ \ \bar c\,b'\models \mathfrak{r}\otimes\mathfrak{q}_{\strok A}$$
which contradicts Proposition \ref{Pwt1}. The claim implies
$a'\equiv a\,(A\bar c)$. Now choose $b''$ such that   $a'b''
\equiv ab \,(A\bar c)$. Then $b'\models \mathfrak{q}_{\strok A\bar
c}$ and $b''\nmodels\mathfrak{q}_{\strok A\bar c}$, by regularity of
$(\mathfrak{q},A\,\bar c)$, imply $b'\models \mathfrak{q}_{\strok
Ab''\bar c}$. Hence $(b'', b')$ is a Morley sequence in
$\mathfrak{q}$ over $A$. By Proposition \ref{Pwt1}   at least one
of $a'\models\mathfrak{p}_{\strok Ab'}$ and
$a'\models\mathfrak{p}_{\strok Ab''}$ holds. The first is not
possible because $a'b' \equiv ab \,(A\bar c)$, and the second
because $a'b'' \equiv ab \,(A\bar c)$. A contradiction.

\smallskip (ii) By interpreting  assumptions of the lemma  and what we have just proved in part
(i), we have:
$$    \mathfrak{p}_{\strok A}(x) \cup   \mathfrak{q}_{\strok
A}(y)\cup \{\phi(y,\bar c)\wedge\theta(y,x)\}\cup
\mathfrak{p}_{\strok A\bar c}(x) \ \ \textmd{is inconsistent}$$ By
compactness  there are
$\varphi_{\mathfrak{p}}(x)\in\mathfrak{p}_{\strok A}$,
$\varphi_{\mathfrak{q}}(y)\in \mathfrak{q}_{\strok A}$ and
$\sigma(x,\bar c) \in \mathfrak p_{\strok A\,\bar c}$ such that:
$$  \varphi_{\mathfrak{p}}(x) \wedge \varphi_{\mathfrak{q}}(y)\wedge \phi(y,\bar c)
\wedge \theta(y,x) \vdash \neg\sigma(x,\bar c) $$ Therefore \  $
\exists y( \varphi_{\mathfrak{p}}(x) \wedge
\varphi_{\mathfrak{q}}(y) \wedge \phi(y,\bar c) \wedge
\theta(y,x))\vdash \neg\sigma(x,\bar c)  \,  $  and
$$ \exists y( \varphi_{\mathfrak{p}}(x) \wedge
\varphi_{\mathfrak{q}}(y) \wedge \phi(y,\bar c) \wedge
\theta(y,x))\notin \mathfrak{p}_{\strok A\bar c}(x)\ .$$ This proves
(ii).
\end{proof}

\begin{prop}\label{Porthgstab}
Generic stability is preserved under non-orthogonality
of symmetric, regular types.
\end{prop}
\begin{proof}
Suppose that $\mathfrak{p}$ and  $\mathfrak{q}$ are   both
regular, non-orthogonal    and that $\mathfrak{p}$ is generically
stable. Choose a small model $M$ such that both $\mathfrak{p}$
and $\mathfrak{q}$ are regular over $M$ and $p\nwor q$ holds for their
corresponding restrictions. Let  $a\models p$, $b\models q$ and
$\theta(y,x)\in\tp(b,a/M)$   be such that
$\theta(b,x)\notin\mathfrak{p}$.

Suppose for a contradiction that $\mathfrak{q}$ is not generically
stable. Then for a suitably chosen larger $M$, $\phi(y,\bar z)$
over $M$,  and a  Morley sequence $\{ b_i\,|\,i\in\omega+\omega\}$
in $\mathfrak{q}$ over $M$ there exists $\bar c$ realizing
\begin{center} $ \{\phi(b_n,\bar z)\,|\,i\in\omega\}\cup
\{\neg\phi (b_{\omega+n},\bar z)\,|\,i\in\omega\} $
\end{center}
Since $\mathfrak{q}$ is symmetric, after possibly replacing the
first and the second $\omega$-part of the sequence,  we may assume
$\neg\phi(x,\bar c)\in \mathfrak{q}$. Also, after replacing the
second $\omega$-part by a Morley sequence in $\mathfrak{q}$ over
$M\,\bar b_{<\omega}\,\bar c$ we may assume that each
$b_{\omega+n}$ realizes $\mathfrak{q}_{\strok M\bar
b_{<\omega+n}\bar c}$.

\smallskip
For each $i\in\omega+\omega$ choose $a_i$ such that $a_ib_i\equiv
ab\,(M)$. Then $\theta(b_i,x)\notin\mathfrak{p}$ witnesses that
$a_i\nmodels \mathfrak p_{\strok Mb_i}$. We claim that:
\begin{equation}\label{e171}   a_i\models \mathfrak{p}_{\strok M\bar b_{<i}}  \   \textmd{holds for all
$i\in\omega+\omega$.}
\end{equation}
To prove it note that $\bar b_{<i}b_i\models
(\mathfrak{p}^{<i}\otimes \mathfrak{p})_{\strok M}$ so, by
Proposition \ref{Pwt1},   at least one of
$a_i\models\mathfrak{p}_{\strok M\bar b_{<i}}$ and $a_i\models
\mathfrak{p}_{\strok Mb_i}$ holds.  Since $\models\theta(b_i,a_i)$
implies $a_i\nmodels\mathfrak{p}_{\strok Mb_i}$ we conclude that
$a_i\models \mathfrak{p}_{\strok M\bar b_{<i}}$ holds, proving the
claim. Combining (\ref{e171}) with $\bar
a_{<i}\subseteq\cl_{\mathfrak{p},M}(\bar b_{<i})$ and the
regularity of $(\mathfrak{p}, M\,\bar b_{<i}) $ we derive:
\begin{equation}\label{e172} a_i\models \mathfrak{p}_{\strok M\bar b_{<i}\bar a_{<i}}  \ \textmd{ holds for all $i\in\omega+\omega$}..
\end{equation}
In particular, $(a_i\,|\,i\in\omega+\omega)$ is a Morley sequence
in $\mathfrak{p}$ over $M$.

 Continuing the proof of the
proposition we first note that $a_0,b_0$ and $\bar c$ satisfy
assumptions of the Lemma \ref{L16}: let $\mathfrak{r}$ be any
global coheir of $\tp(\bar c/M)$.  Then $b_0\nmodels
\mathfrak{q}_{\strok M\bar c}$ (witnessed by $\models \phi(b_0,\bar
c)$)  and $a_0\nmodels \mathfrak{p}_{\strok Mb}$ (witnessed by
$\models\theta(b_0,a_0)$). So we apply Lemma \ref{L16}(ii) and
consider the formula
$$ \exists y(\phi(y,\bar c) \wedge \varphi_{\mathfrak{q}}(y) \wedge \varphi_{\mathfrak{p}}(x)\wedge
\theta(y,x))\notin \mathfrak{p} \ .$$ Denote it  by $\psi(x,\bar
c)$.   Then $\models\psi(a_n,\bar c)$ holds for all $n\in \omega$:
the existential quantifier is witnessed by $b_n$.  On the other
hand, $a_{\omega+n}\models \mathfrak{p}_{\strok M\bar c}$ implies
that $\models\neg\psi(a_{\omega+n},\bar c)$ holds for all
$n\in\omega$. Therefore, $\bar c$ realizes
\begin{center} $
\{\psi(a_n,\bar z)\,|\,n\in\omega\}\cup \{\neg\psi
(a_{\omega+n},\bar z)\,|\,n\in\omega\} $
\end{center}
and $\mathfrak{p}$ is not generically stable. A contradiction.
\end{proof}

\noindent{\em Proof of Theorem  \ref{Tgenstabnor}.} It remains to
prove that $\nor$  is an equivalence relation on the set of all
generically stable regular types.   Only transitivity needs verification, so
assume that $\mathfrak{p}\nor \mathfrak q$ and $\mathfrak  q\nor\mathfrak r$ are  regular and generically stable. Let
$a,b, c$   realize  $\mathfrak p,\mathfrak q,\mathfrak r $  respectively be  such that   $b\nmodels
\mathfrak{q}_{\strok \bar M  c}$  and $a\nmodels \mathfrak{p}_{\strok
\bar M b}$. Then, by Lemma \ref{L16}(i),  $a\nmodels \mathfrak{p}_{\strok \bar M c}$ holds, so $\mathfrak p\nor \mathfrak r$ and $\nor$ is transitive. \qed

\section{Strong regularity}\label{Sstrong}

In this section we study non-orthogonality of an invariant type
and a strongly regular type. We will show that the dependence of
their realizations is witnessed by semi-isolation.

\begin{lem}\label{Lsemi} Suppose that $(\mathfrak{p},\phi_{\mathfrak p})$ is $A$-invariant, definable and strongly regular and that $a,b$ are realizations of $p=\mathfrak p_{\strok A}$.

\smallskip (i) If $\tp(a/Ab)$ is finitely satisfiable in
$\mathcal{C}_{\mathfrak{p}}=\phi_{\mathfrak{p}}(\bar M) \smallsetminus \mathfrak{p}_{\strok
A}(\bar M)$  then $(a,b)$ is a Morley sequence in
$\mathfrak{p}$ over $A$.

\smallskip (ii) If $(a_0,...,a_n)$ is a  $\mathcal{C}_{\mathfrak{p}}$-sequence of realizations of $p$, then $(a_n,...,a_0)$ is a Morley sequence in $\mathfrak p$ over $A$.

\smallskip (iii)  If $p$ is non-isolated then: \ \  $b\nmodels \mathfrak{p}_{\strok Aa}$ \ iff \ $a\in\Sem_A(b)$.
\end{lem}

\begin{proof} (i) Assuming   that $(a,b)$ is not a Morley sequence we will show that $\tp(a/Ab)$ is not a $\mathcal{C}_{\mathfrak{p}}$-type. Let $\varphi(x,y)\in\tp(a,b/A)$ witness   $b\nmodels \mathfrak p_{\strok A\,a}$; then $\models \varphi(a,b)\wedge d_{\mathfrak p}t\varphi(a,t)$ holds.
We claim that $\varphi(x,b)\wedge d_{\mathfrak p}t\varphi(x,t)$ is not satisfied in $\mathcal{C}_{\mathfrak{p}}$: otherwise, for some $c\in \mathcal{C}_{\mathfrak{p}}$  we would have
$\models \varphi(c,b)\wedge d_{\mathfrak p}t\varphi(c,t)$ which implies $b\nmodels \mathfrak p_{\strok A\,c}$ and $p\nwor \tp(c/A)$. This is impossible because, by Lemma \ref{Lworinv}, $(p,\phi_{\mathfrak p})$  satisfies  WOR.

\smallskip (ii) Follows from part (i)  by induction.

\smallskip (iii) To prove the $\Rightarrow )$ part  assume $b\nmodels \mathfrak{p}_{\strok Aa}$. Then $(a,b)$ is not a Morley sequence and, by part (i), $\tp(a/Ab)$ is not a
$\mathcal{C}_{\mathfrak{p}}$-type. Choose $\theta(x,b)\in\tp(a/Ab)$ which is not satisfied in $\mathcal{C}_{\mathfrak{p}}$. Then $\theta(x,b)\wedge \phi_{\mathfrak p}(x)\vdash p(x)$ witnesses  $a\in\Sem_A(b)$. This proves the $\Rightarrow)$ part.

\smallskip
For the $\Leftarrow)$ part assume $b\models \mathfrak{p}_{\strok Aa}$. Then $(a,b)$ is a Morley sequence over $A$. Since $p$ is non-isolated, it is finitely satisfiable in $\mathcal{C}_{\mathfrak{p}}$ so,
by Fact \ref{Ffs}, it has an extension in $S_1(Ab)$ which is is finitely satisfiable in   $\mathcal{C}_{\mathfrak{p}}$;  let $a'$ realize it. By part (i) $(a',b)$ is a Morley sequence over $A$ so  $\tp(a,b/A)=\tp(a',b/A)$. Hence $\tp(a/Ab)$ is finitely satisfiable in $\mathcal{C}_{\mathfrak{p}}$ and $a\notin\Sem_A(b)$.
\end{proof}

\begin{prop}\label{Psem}
Suppose that $(\mathfrak{p},\phi_{\mathfrak p})$ is $A$-invariant, strongly regular and generically stable.
Further, suppose that $\mathfrak{q}$ is $A$-invariant and that  $a\models \mathfrak{p}_{\strok A}$ and $\bar b\models \mathfrak{q}_{\strok A}$ are such that $a\nmodels \mathfrak{p}_{\strok A\bar b}$. Then $\bar b$ semi-isolates $a$ over $A$.
\end{prop}
\begin{proof}  Let $p,q$ denote $\mathfrak{p}_{\strok A}$ and
$\mathfrak{q}_{\strok A}$ respectively.   Choose   $\theta(x,\bar
b)\in \tp(a/A\bar b)$ witnessing  $a\nmodels \mathfrak{p}_{\strok
A\bar b}$ and, without loss of generality, assume $\models
\theta(x,\bar b)\Rightarrow \phi_{\mathfrak p}(x)$. Suppose that
the conclusion of the proposition fails: $\bar b$ does not
semi-isolate $a$ over $A$.   Then   $\tp(a/A\bar b)$ is finitely
satisfiable in $\mathcal C= \theta(\bar M,\bar b)\smallsetminus
p(\bar M)$. Let  $(a_i\,|\,i\leq n)$ be a $\mathcal C$-sequence of
realizations of  $\tp(a/A\bar b)$. Since $\mathcal C\subseteq
\mathcal C_{\mathfrak p}= \phi_{\mathfrak p}(\bar M)\smallsetminus
p(\bar M)$ $(a_i\,|\,i\leq n)$ is also a $\mathcal C_{\mathfrak
p}$-sequence.   By Lemma \ref{Lsemi} 
$(a_n,a_{n-1},...,a_0)$ is a Morley sequence in $\mathfrak{p}$
over $A$. Since   $\models \theta(a_i,\bar b)$ holds for all
$i\leq n$ 
$$(\mathfrak p^n)_{\strok A}(x_1,...,x_n)\cup
\{ \theta(x_i,\bar b)\,|\,a\leq i\leq n\}$$ is consistent for all
$n$. Let $(a_i'\,|\,i\in\omega)$ be an infinite Morley sequence in
$\mathfrak{p}$ over $A$. By compactness \
 $q(\bar y)\cup\{\theta(a_i',\bar y)\,|\,i\in \omega\}$  \
is consistent so, without loss of generality, assume that $\bar b$ realizes that type. Let
$(a_{\omega+i}'\,|\,i\in\omega)$ be an infinite Morley sequence in $\mathfrak{p}$ over $A\bar a_{<\omega}'\bar b$. Then $\neg\theta(x,\bar b)\in\mathfrak{p}$ implies  that
$\models\neg\theta(a_{\omega+i}', \bar b)$ holds for all $i\in\omega$. Therefore $\bar b$ realizes
$$q(\bar y)\cup\{\theta(a_i',\bar y)\,|\,i\in \omega\}\cup\{\neg\theta(a_{\omega+i}',\bar y)\,|\,i\in \omega\}$$
and $\mathfrak{p}$ is not generically stable. A contradiction.
\end{proof}

\begin{lem}\label{Lnorrk}
  $\mathfrak{p}\leq_{RK}\mathfrak{q}$ \  implies \ $\mathfrak{p}\nor \mathfrak{q}$.
\end{lem}
\begin{proof} Choose $\bar a$ and $\bar b$  realizing $\mathfrak{p}$ and $\mathfrak{q}$ respectively   such that $\bar a\in\Sem_{\bar M}(\bar b)$; then $\tp(\bar a/\bar M\bar b)$ is not a coheir of $\mathfrak{p}$. Therefore  $\tp(\bar a/\bar M\bar b)$ and a coheir are two distinct global extensions of $\mathfrak{p}(x)$, so $\mathfrak{p}\nor \mathfrak{q}$ holds.
\end{proof}

The following is a   technical version of Theorem
\ref{Tstrregintro}:

\begin{thm1}\label{Trk}
(1) Strongly regular, generically stable types are minimal in the Rudin-Keisler global order. Moreover,
if $\mathfrak p$ is strongly regular and  generically stable   then for any $\mathfrak q$: \  $\mathfrak p\nor \mathfrak q$ iff $\mathfrak p\leq_{RK} \mathfrak q$.

\smallskip (2) Non-orthogonal strongly regular, generically stable types are strongly $RK$-equivalent. Moreover, whenever $a,b$ realize over $\bar M$ such types $\mathfrak p$ and $\mathfrak q$  and $a\nmodels \mathfrak p_{\strok \bar M b}$ holds, then:  $b\nmodels \mathfrak q_{\strok \bar M a}$, $b\in \Sem_{\bar M}(a)$ and $a\in \Sem_{\bar M}(b)$.
\end{thm1}
\begin{proof} (1) Suppose that $\mathfrak p$ is strongly regular and generically stable. Then, by Lemma \ref{Lnorrk}, $\mathfrak{p}\leq_{RK}\mathfrak{q}$    implies   $\mathfrak{p}\nor \mathfrak{q}$. To prove the other direction assume   $\mathfrak q\nor \mathfrak p$ and let $a\models \mathfrak p$ and $\bar b\models \mathfrak q$ be such that $a\nmodels \mathfrak p_{\strok \bar M \bar b}$. Then, by Proposition \ref{Psem}, $a\in \Sem_{\bar M}(\bar b)$ holds, so $\mathfrak p\leq_{RK}\mathfrak q$. This proves the other implication.

\smallskip (2) Suppose that both $\mathfrak p$ and $\mathfrak q$ are strongly regular,
generically stable and that $a\models \mathfrak p$ and $b\models
\mathfrak q$ are such that $a\nmodels \mathfrak p_{\strok \bar M
b}$. Then, by Proposition \ref{Psem}, $a\in \Sem_{\bar M}(b)$
holds. Let $a_1\models \mathfrak p$ be such that   $b\nmodels
\mathfrak p_{\strok \bar M a_1}$. By Proposition \ref{Psem} again
we have $b\in\Sem_{\bar M}(a_1)$. By transitivity $a\in\Sem_{\bar
M}(a_1)$ so $(a_1,a)$ is not a Morley sequence in $\mathfrak p$.
Then, by symmetry, $(a,a_1)$ is not a Morley sequence   so
$a_1\nmodels \mathfrak p_{\strok \bar Ma}$.  By Lemma
\ref{Lsemi}(iii) $a_1\in\Sem_{\bar M}(a)$ and, by transitivity,
$b\in\Sem_{\bar M}(a)$ holds.  This proves the "moreover" part 
and  strong RK-equivalence of $\mathfrak p$ and $\mathfrak q$
follows.
\end{proof}
We have just proved that strongly regular types are RK-minimal.
The converse is not  true: Take the theory of the random graph.
There distinct global 1-types are $\leq_{RK}$-incomparable, so every 1-type is
RK-minimal. However, none of them is strongly regular. We now
proceed towards proving    Theorem \ref{Tlocal}.
\begin{lem}\label{L35}
Suppose that $(\mathfrak{p}(x),\phi_{\mathfrak{p}}(x))$ and $(\mathfrak{q}(x),\phi_{\mathfrak{q}}(x))$ are  strongly regular, generically stable  and non-orthogonal. Let  $\mathfrak{r}\neq \mathfrak{q}$   be any global type containing $\phi_{\mathfrak{q}}(x)$. Then
$\mathfrak{p}\perp \mathfrak{r}$.
\end{lem}
\begin{proof}Suppose   that the conclusion fails: $\mathfrak{p}\nor \mathfrak{r}$.  Theorem \ref{Trk}\,(1) implies $\mathfrak{p}\leq_{RK} \mathfrak{r}$.   $\mathfrak{p}\nor \mathfrak{q}$, by  Theorem \ref{Trk}\,(2), implies $\mathfrak{p}\equiv_{RK}\mathfrak{q}$.
Combining the two we conclude $\mathfrak{q}\leq_{RK}\mathfrak{r}$ and, by Lemma \ref{Lnorrk}, $\mathfrak q\nor \mathfrak r$, contradicting   the strong regularity of $(\mathfrak{q}(x),\phi_{\mathfrak{q}}(x))$.
\end{proof}

\noindent{\em Proof of Theorem \ref{Tlocal}.} \ Suppose that
$(\mathfrak{p}(x),\phi_{\mathfrak{p}}(x))$ and
$(\mathfrak{q}(x),\phi_{\mathfrak{q}}(x))$ are $M$-invariant,
strongly regular and generically stable.  We will prove that  the
following conditions are all equivalent: \

(1) \ \  $\mathfrak p\nor\mathfrak q$;

\smallskip (2) \ \ $\mathfrak{p}_{\strok M}\nwor \mathfrak{q}_{\strok M}$;

\smallskip (3)  \ \ For all $C\supseteq  M$: \ $\mathfrak{p}_{\strok
C}\nwor \mathfrak{q}_{\strok C}$.

\noindent (2)$\Rightarrow$(3) holds  by  Fact  \ref{For} (or by
Proposition \ref{Pwor=or}), and (3)$\Rightarrow$(1) is obvious. We
will  prove (1)$\Rightarrow$(2). So assume $\mathfrak p\nor
\mathfrak q$ and let $\bar c$ be such that $\mathfrak{p}_{\strok
M\bar c}\nwor \mathfrak{q}_{\strok M\bar c}$. Choose $a\models
\mathfrak{p}_{\strok M\bar c}$ and  $b\models \mathfrak{q}_{\strok
M\bar c}$ such that $a\nmodels \mathfrak{p}_{\strok M\bar cb}$.
Suppose that $\varphi(x,b,\bar c)\notin \mathfrak{p}_{\strok
Mb\bar c}$ is satisfied by $a$. Then:
$$\models d_{\mathfrak{p}}x\exists y(\phi_{\mathfrak{q}}(y)\wedge
\varphi(x,y,\bar c)\wedge \neg d_{\mathfrak{p}}t\varphi(t,y,\bar
c)) .$$ Let $\bar c'\in M$ satisfy the formula in place of $\bar
c$. Then:
$$\models \exists y(\phi_{\mathfrak{q}}(y)\wedge
\varphi(a,y,\bar c')\wedge \neg d_{\mathfrak{p}}t\varphi(t,y,\bar
c')) .$$ Let $\psi(x,y)$ be the formula  \   $
\phi_{\mathfrak{q}}(y)\wedge \varphi(x,y,\bar c')\wedge \neg
d_{\mathfrak{p}}t\varphi(t,y,\bar c')$. \ We {\em claim}
$\psi(a,y)\vdash \mathfrak q_{\strok M}(y)$. Suppose for a
contradiction that $b'$ satisfies  $\models \psi(a,b')$ and
$\tp(b'/M)\neq \mathfrak q_{\strok M}$. Choose a formula
$\theta(y)\in\tp(b'/M)$ which is not in $q$ and implies
$\phi_{\mathfrak q}(y)$. Then $\models \psi(a,b')\wedge
\theta(b')$ so $\exists y(\psi(x,y)\wedge \theta(y)\in\mathfrak
p$. Let $a_0\models \mathfrak p$ and let $b_0$ be such that
$\models \psi(a_0,b_0)\wedge \theta(b_0)$.  \ $\models
\psi(a_0,b_0)$ implies $\models \varphi(a_0,b_0,\bar c')\wedge
\neg d_{\mathfrak{p}}t\varphi(t,b_0,\bar c')$ so $a_0\nmodels
\mathfrak p_{\strok\bar Mb_0}$. Therefore $\mathfrak p\nor \mathfrak
r$, where $\mathfrak r=\tp(b_0/\bar M) $. Now $\theta(y)\in
\mathfrak r$ implies $\phi_{\mathfrak q}(y)\in \mathfrak r\neq
\mathfrak q$ so, by Lemma \ref{L35}, $\mathfrak p\perp \mathfrak
r$. A contradiction. Thus $\psi(a,y)\vdash \mathfrak q_{\strok
M}(y)$.

To finish the proof we note that $\mathfrak q_{\strok M}(y)$ has at
least two distinct extensions in $S(Ma)$: one that contains
$\psi(a,y)$ and the coheir. Therefore $\mathfrak p_{\strok M}\nwor
\mathfrak q_{\strok M}$. \qed

\section{Omitting types}\label{Somit}

\begin{lem}\label{Lnonisol}
Suppose that $\mathfrak p$ is generically stable and  regular over $A$, 
and that $\mathfrak p_{\strok A}$ is non-isolated. Then $\mathfrak
p_{\strok B}$ is non-isolated for all $B\supseteq A$.
\end{lem}
\begin{proof}
Otherwise  there is $\bar b$ such that $\mathfrak p_{\strok A\bar b}$
is isolated, by $\varphi(x,\bar b)$ say. Since $\mathfrak p_{\strok A}$ is non-isolated it has a non-isolated extension in $S_1(A\bar b)$, so  \
$\mathfrak p_{\strok A}(x)\cup\{\neg\varphi(x,\bar b)\}$ \ is
consistent.  Let $n$ be  the length of the longest possible  Morley
sequence $a_1,a_2,...,a_n$ in $\mathfrak p$ over $A$ satisfying
$\models \bigwedge_{i=1}^{n}\neg\varphi(a_i,\bar b)$; it exists because $\varphi(x,\bar b)\in \mathfrak p$ and
$\mathfrak p$ is generically stable.  Since $\mathfrak p$ is regular over $A\bar
b$  and none of the $a_i$'s realize $\mathfrak p_{\strok
A\bar b}$ we have $\mathfrak p_{\strok A\bar b}(x)\vdash\mathfrak
p_{\strok A\bar ba_1...a_n}(x)$. Let  $a$ realize $\mathfrak p_{\strok
Aa_1...a_n}$. The maximality of $n$ implies \ $\tp_{\bar y}(\bar
b/A) \cup \{\bigwedge_{i=1}^{n}\neg\varphi(a_i,\bar y)\}\vdash
\varphi(a, \bar y)$. Let $\phi(\bar y)\in \tp_{\bar y}(\bar b/A)$
be such that:
$$\models \forall\bar y((\phi(\bar y)\wedge\bigwedge_{i=1}^{n}\neg\varphi(a_i,\bar
y) )\Rightarrow \varphi(a,\bar y))$$ Denote this formula  by
$\psi(a,a_1,...,a_n)$. Then $\models d_{\mathfrak
p^n}z_1...z_n\psi(a,z_1,...,z_n)$.  We will reach the
contradiction by showing that $d_{\mathfrak
p^n}z_1...z_n\psi(x,z_1,...,z_n)$ isolates  $\mathfrak p_{\strok
A}$. So let $a'$ satisfy the formula and, without loss of
generality, assume that $(a_1,...,a_n)$ witnesses the $d_{\mathfrak
p^n}$ quantifier: $a_1...a_n\models\mathfrak p^n_{\strok Aa'}$.
Then: \setcounter{equation}{0}
\begin{equation}\label{erc}\models \forall\bar y((\phi(\bar
y)\wedge\bigwedge_{i=1}^{n}\neg\varphi(a_i,\bar y) )\Rightarrow
\varphi(a',\bar y))\end{equation} Let $\bar b'$ be such that $\bar
b'\equiv \bar b\,(Aa_1...a_n)$. Then the left hand side of the
implication in (\ref{erc}) is satisfied by $\bar b'$, so we derive
$\models\varphi(a',\bar b')$. Since $\varphi(x,\bar b')$ isolates
$\mathfrak p_{\strok A\bar b'}$ we conclude   $a'\models\mathfrak
p_{\strok A\bar b'}$ and, in  particular, $a'\models\mathfrak p_{\strok
A}$.
\end{proof}

\begin{prop}\label{Pwor=or}
Suppose that $\mathfrak p$ is generically stable and regular over $A$, and that
$\mathfrak q$ is $A$-invariant. Then $\mathfrak p\perp \mathfrak q$ implies  $\mathfrak
p_{\strok C}\wor \mathfrak q_{\strok C}$ for all $C\supseteq A$.
\end{prop}
\begin{proof}
Assume  $\mathfrak p_{\strok A} \nwor \mathfrak q_{\strok A}$ and let us prove
that $\mathfrak p_{\strok A\bar c} \nwor
\mathfrak q_{\strok A\bar c}$ holds for all tuples $\bar c$.
Witness $\mathfrak p_{\strok A} \nwor \mathfrak q_{\strok A}$ by
$a\models \mathfrak p_{\strok A}$, $\bar b\models\mathfrak
q_{\strok A}$ and $\varphi(x,\bar y)\in \tp(a,\bar b/A)$ such that
$\varphi(x,\bar b)\notin\mathfrak p_{\strok A\bar b}$.

For each $\phi(x,\bar c)\in \mathfrak p_{\strok A\bar c}$ pick a
Morley sequence of maximal possible length $(a^{\phi}_i|i\leq
n_{\phi})$ in $\mathfrak p$ over $A$ such that $\models
\bigwedge_{i}\neg\phi(a^{\phi}_i,\bar c)$. Since $\mathfrak p$ is
generically stable each of the chosen sequences is finite. Let $D$
be the union of all them and, without loss of generality, assume
$\bar b\models \mathfrak q_{\strok AD\bar c}$. We {\em claim} that
$a$ realizes $\mathfrak p_{\strok AD}$. Otherwise, there would be
$\bar d\in D$ such that $a\nmodels \mathfrak p_{\strok A\bar d}$.
Moreover, such a $\bar d$ can be chosen $\cl_{\mathfrak
p,A}$-independent, i.e. realizing a Morley sequence  in $\mathfrak
p$ over $A$. Then we would have:
\begin{center}
 $\bar d,\bar b\models \mathfrak p^n\otimes\mathfrak q_{\strok A}$, \ $a\nmodels \mathfrak p_{\strok A\bar d}$ \ and \ $a\nmodels \mathfrak p_{\strok A\bar b}$
\end{center}
which contradicts Proposition \ref{Pwt1} and proves the claim: $a\models \mathfrak p_{\strok AD}$.

To complete the proof it remains to note that $\mathfrak p_{\strok AD}\vdash \mathfrak p_{\strok A\bar c}$ holds by our choice of $D$; thus $a\models \mathfrak p_{\strok A\bar c}$. Summing up, we have: $\bar b\models \mathfrak q_{\strok A\bar c}$, $a\models \mathfrak p_{\strok A\bar c}$ and $a\nmodels\mathfrak p_{\strok A\bar b}$. Thus $\mathfrak p_{\strok A\bar c} \nwor
\mathfrak q_{\strok A\bar c}$. 
\end{proof}

\begin{cor}\label{Cworomega}
Suppose that $\mathfrak p$ is generically stable and regular over $A$,
and that $\mathfrak q$ is $A$-invariant. Then $\mathfrak p\perp
\mathfrak q$ implies \ $\mathfrak p^{\omega}_{\strok C}\wor
\mathfrak q^{\omega}_{\strok C}$ \ for all $C\supseteq A$.
\end{cor}
\begin{proof}
Easy induction using Proposition \ref{Pwor=or}.
\end{proof}

\noindent{\em  Proof of Theorem \ref{Tomit}.} \ Suppose that $A$ is countable and that $\{\mathfrak
 p_i \,|\,i\in I\}$ is a countable family of pairwise
orthogonal, regular over $A$, generically stable types. Assume that each $\mathfrak p_{i\,\strok A}$ is non-isolated. Let $f:I\longrightarrow \omega$. We will prove that there is a countable
$M_f\supseteq A$ such that $\dim_{\mathfrak p_i}(M_f/A)=f(i)$ for all $i\in I$.

\smallskip
 For each $i\in I$ for which $f(i)\neq 0$ holds choose a Morley sequence
$J_i=(a_j^i\,|\,1\leq j\leq f(i))$ in $\mathfrak p_i$ over $A$.
Let $J$ be the union of all the chosen sequences. By Lemma
\ref{Lnonisol} each $p_i=\mathfrak p_{i\,\strok AJ}$ is
non-isolated so, by the Omitting Types Theorem,  there is a
countable $M_f\supseteq AJ$ which omits all the $p_i$'s. We will
prove that $M_f$ is the desired model.  $J_i\subseteq M_f$ implies $\dim_{\mathfrak p_i}(M_f/A)\geq f(i)$. To prove that the equality holds, it suffices to show that each $\mathfrak p_{i\,\strok AJ_i}$ is omitted in $M_f$. By repeatedly applying
Corollary \ref{Cworomega}  we get $\mathfrak p_{i\,\strok
AJ_i}\vdash p_i$. Thus any realization of $\mathfrak p_{i\,\strok
AJ_i}$ in $M_f$ also realizes $p_i$. The latter type is omitted in
$M_f$, so $\mathfrak p_{i\,\strok AJ_i}$ is omitted in $M_f$, too.
Therefore  $\dim_{\mathfrak p_i}(M_f/A)= f(i)$ completing the
proof of the theorem.\qed

\end{document}